\documentclass[11pt,a4paper]{amsart}

\usepackage{amsthm, amsfonts, amssymb, amsmath, latexsym, enumerate, array}
\usepackage[all]{xy}
\usepackage[utf8]{inputenc}
\usepackage{hyperref,cite,mdwlist,mathrsfs}

\newcommand{\mathbold}{\mathbf}

\newtheorem*{main*}{Theorem}
\newtheorem*{mcorol*}{Corollary}
\newtheorem{thm}{Theorem}[section] \newtheorem{lem}[thm]{Lemma}
 \newtheorem{prop}[thm]{Proposition}
\newtheorem{claim}[thm]{Claim}
\newtheorem*{thm*}{Theorem}

\theoremstyle{definition} 
\newtheorem{eg}[thm]{Example}

\newtheorem*{rmk*}{Remark}

\newcommand{\OO}{\mathscr{O}}

\newcommand{\EExt}{\mathscr{E}xt}
\newcommand{\HHom}{\mathscr{H}om}

\newcommand{\sE}{\mathscr{E}}

\newcommand{\sH}{\mathscr{H}}
\newcommand{\sF}{\mathscr{F}}
\newcommand{\sK}{\mathscr{K}}
\newcommand{\sL}{\mathscr{L}}
\newcommand{\sP}{\mathscr{P}}
\newcommand{\sM}{\mathscr{M}}
\newcommand{\sN}{\mathscr{N}}
\newcommand{\sC}{\mathscr{C}}

\newcommand{\sB}{\mathscr{B}}
\newcommand{\sA}{\mathscr{A}}
\newcommand{\sT}{\mathscr{T}}
\newcommand{\sY}{\mathscr{Y}}
\newcommand{\sG}{\mathscr{G}}

\DeclareMathOperator{\id}{{\sf id}}
\DeclareMathOperator{\rD}{D}
\DeclareMathOperator{\sing}{{\mathrm Sing}}
\DeclareMathOperator{\cl}{{\mathrm Cl}}

\DeclareMathOperator{\Hilb}{Hilb}

\DeclareMathOperator{\Hom}{Hom} 
\DeclareMathOperator{\im}{Im} \DeclareMathOperator{\cok}{coker}

\DeclareMathOperator{\HH}{H} \DeclareMathOperator{\hh}{h}

\DeclareMathOperator{\Pic}{Pic}

\DeclareMathOperator{\Aut}{Aut}
\DeclareMathOperator{\codim}{codim}

\DeclareMathOperator{\Coh}{Coh}
\DeclareMathOperator{\Spec}{Spec}

\newcommand{\Z}{\mathbb Z} 
 
 \newcommand{\p}{\mathbb P}
\newcommand{\V}{\mathbb V}
 \newcommand{\G}{\mathbf G}

\newcommand{\RR}{\mathbold{R}}
\newcommand{\bW}{\mathbold{W}}

\DeclareMathOperator{\ts}{\otimes}
\newcommand{\mono}{\hookrightarrow}

\newcommand{\xr}{\xrightarrow}
\newcommand{\kk}{\boldsymbol{k}}

\SelectTips{cm}{} \newdir{ >}{{}*!/-5pt/@{>}}

\allowdisplaybreaks[1]
\begin{document}


\title[Hilbert scheme of determinantal subvarieties]{On the Hilbert
  scheme of varieties \\ defined by maximal minors}

\author{Daniele Faenzi}
\email{{\tt daniele.faenzi@univ-pau.fr}}
\address{Universit\'e de Pau et des Pays de l'Adour \\
  Avenue de l'Universit\'e - BP 576 - 64012 PAU Cedex - France}
\urladdr{{\url{http://univ-pau.fr/~faenzi/}}}
\author{Maria Lucia Fania}
\email{{\tt fania@univaq.it}}
\address{Universit\`{a} degli Studi dell'Aquila \\
Via Vetoio Loc. Coppito\\ 67100 L'Aquila\\ Italy}
\urladdr{{\url{http://univaq.it/~fania/}}}



\thanks{D.F. partially supported by ANR contract Interlow
  ANR-09-JCJC-0097-0 and ANR GEOLMI.
M.L.F. partially supported by MIUR funds, PRIN project \lq\lq\thinspace Geometria delle
Variet\`a Algebriche''.}

\begin{abstract}
We compute the dimension of the Hilbert scheme of subvarieties
of positive dimension in
projective space which are cut by maximal minors
of a matrix with polynomial entries.
\end{abstract}

\maketitle

\section{Introduction}

A determinantal subvariety $X$ of the projective space  $\p^n$ is the
locus defined by the vanishing of all 
minors of a given order of a matrix $M$ of
homogeneous polynomials. 
Many classical varieties can be constructed in this way, for instance
Segre and Veronese varieties, rational normal scrolls, Palatini
scrolls, Bordiga varieties and so forth.
The literature on the subject is rather vast, let us refer to the 
monographs \cite{northcott:finite-free, bruns-vetter:LNM,
weyman:tract,miro-roig:determinantal} and \cite[Lecture 9]{harris:ag}
for more on these classical objects.

If one attempts to parametrize all determinantal varieties of a given
type (i.e. for fixed degrees of the entries of $M$), a first step is
to look at $[X]$ as a point of a component $\sH$ of the Hilbert scheme of subscheme in $\p^n$, and study to what
extent the family of determinantal varieties fills in $\sH$.
In this spirit, Ellingsrud proved in \cite{ellingsrud:hilbert}
that determinantal varieties defined by maximal minors, in case of codimension $2$ and dimension $\ge 1$ are unobstructed and
their family is open and dense in $\sH$.
In the series of papers \cite{kleppe-migliore-miro-roig-nagel-peterson,
  kleppe-miro-roig:dimension, kleppe-miro-roig:families-PAMS}, 
the same behavior was established in many more cases, leading to
conjecture that this phenomenon should be general.

The goal of  this paper is to prove this conjecture in general, for determinantal
varieties defined by maximal minors, in the range of dimension at
least $1$ and codimension at least $2$.

\medskip

To state our main result properly we adopt now a more precise
language.
Given integers 
$\alpha_1 \leq \cdots \leq \alpha_a$,
$\beta_1 \leq \cdots \leq \beta_b$, and assuming $a \ge b$,
we consider a matrix  $M=(M_{i,j})$ of homogeneous polynomials of
degree $\alpha_j-\beta_i$ in $n+1$ variables, and 
we let $X$ be defined by all minors of order $b$ of $M$.
The subvariety $X$ sits in $\p^n$, and we let $p(t)$ be the Hilbert
polynomial of $X$.
Let us denote by $[X]$ the class of $X$ in the Hilbert scheme $\Hilb_{p(t)}({\p^n})$ parametrizing subschemes of
${\p^n}$ having Hilbert polynomial $p(t)$, and by $\sH$ the
irreducible component (or their union if there are more than one) of $\Hilb_{p(t)}({\p^n})$ containing $[X]$.
To parametrize all varieties defined by maximal minors, 
we define the bundles
$\sA = \oplus_{j=1,\ldots,a} \OO_{{\p^n}}(-\alpha_j)$, and $\sB =
\oplus_{i=1,\ldots,b} \OO_{{\p^n}}(-\beta_i)$, and we consider the vector space $\bW$ and the algebraic group $\G$:
\[
\bW = \Hom_{{\p^n}}(\sA,\sB), \qquad \G = \Aut_{\p^n}(\sA) \times \Aut_{\p^n}(\sB).
\]
A matrix $M$ corresponds this way to an element $\phi$ of $\bW$, and the associated
variety $X$ is the first degeneracy locus of $\phi$, hence we set
$M=M_\phi$, $X=X_\phi$.
An element $\rho = (g,h) \in \G$ acts on $\bW$ by
$\rho.\phi = g \circ \phi \circ h^{-1}$.
Of course, all elements in the $\G$-orbit $[\phi]$ of $\phi$ give the same
degeneracy locus $\phi$, since the ideal generated by minors of a
given order is invariant under change of basis.
There exists an open subset $\bW^\circ$ of $\bW$ such that $\bW^\circ /\G$ is a generically
smooth irreducible variety that we 
denote by $\sY$.

\smallskip

We consider thus  the natural rational map:
\begin{align*}
  F : \sY & {\dashrightarrow} \sH, &&  F : [\phi]  \mapsto [X_\phi].
\end{align*}

The main result of this note is the following analysis of the rational map $F$.

\begin{main*} 
  Choose integers $n$, $b \le a-1$, $\alpha_1 \leq \cdots \leq \alpha_a$
  and $\beta_1 \leq \cdots \leq \beta_b$, satisfying the following
  numerical condition:
  \begin{equation}
    \label{numerical}
  \mbox{$\alpha_i \ge \beta_{i+1}$, $\forall i=1,\ldots,b-1$},
  \quad \mbox{and $\alpha_i > \beta_i$, for some $i=1,\ldots,b$.}
  \end{equation}
  If $\phi$ is general enough, then:
  \[
  \dim(X_\phi)=n+b-a-1, \qquad \codim(\sing(X_\phi),X_\phi)\ge 3.
  \]
  Moreover the following holds:
  \begin{enumerate}[i)]
  \item \label{main-i} if $n+b-a-1 \ge 1$, then $F$ is
    generically finite, so  $\dim(\im(F))=\dim(\sY)$;
  \item \label{main-ii} if $n+b-a-1 \ge 2$, then $F$ is also dominant,
    in particular $\sH$ is an irreducible, generically smooth 
  variety and $\dim(\sH)=\dim(\sY)$; 
  \item \label{main-iii} if $\alpha_1>\beta_b$  and $n+b-a-1 \ge 2$, 
  then $F$ is birational.
  \end{enumerate}
\end{main*}

As recalled above, this result was motivated by a conjecture of Kleppe and
Mir\'o-Roig, see \cite[Conjecture
4.2]{kleppe-miro-roig:families-PAMS}, rooted in early work of
Ellingsrud.
This conjecture is solved by part
\eqref{main-ii} above.

What we show is in fact stronger since \eqref{main-iii} proves uniqueness of
determinantal representations in the range $\alpha_1 > \beta_b$.
Further, part \eqref{main-i} above addresses the positive-dimensional
range of \cite[Conjectures 4.1]{kleppe-miro-roig:families-PAMS}.
Our result thus completes 
\cite{kleppe-miro-roig:dimension, 
  kleppe-migliore-miro-roig-nagel-peterson,
  kleppe-miro-roig:families-PAMS,
  kleppe:low-JPAA}, 
where these conjectures are addressed for several ranges of the $\alpha_j$'s
and $\beta_i$'s.
  
  One should be aware that the dimension $\dim(\sY)$ can be calculated
  explicitly as a function of the $\alpha_j,\beta_i$'s and $n$.
  Indeed we set $c=a-b+1$, and, according to
  \cite{kleppe-miro-roig:families-PAMS}, we define $\lambda_c$ as:
  \begin{align} \label{lambdac}
    \lambda_c = & \sum_{\substack{j=1,\ldots,a \\ i=1,\ldots,b}}
  {\alpha_j-\beta_i+n\choose
    n}+{\beta_i-\alpha_{j}+n\choose n} - \\
  \nonumber - & \sum_{i,j=1,\ldots,a}{\alpha_i-\alpha_j+n\choose n}
  -\sum_{i,j=1,\ldots,b}{\beta_{i}-\beta_j+n\choose n}+1.
  \end{align}
  Further we define, for $i=3,\ldots,c$ the integers:
  \[  \ell_i= \sum_{j=1,\ldots,b+i-1} \alpha_j - \sum_{i=1,\ldots,b}
  \beta_i, \qquad h_{i-3}=2 \alpha_{b+i-1}- \ell_i+n.
  \]
  Finally, for $i=0,\ldots,c-3$, the integers $K_{i+3}$ are defined by:
  \begin{equation}
    \label{Ki}
    \sum_{\substack{r+s=i \\ r,s \ge 0}} \sum_{\substack{1\le i_1< \cdots <i_r\le b+i+1, \\
        1\le j_1\le \cdots \le j_s\le b}} (-1)^{i-r}
    {h_i+\alpha_{i_1}+\cdots+\alpha_{i_r}+\beta_{j_1}+\cdots+\beta_{j_s}\choose{n}} 
  \end{equation}
  In these terms, the dimension of $\sY$ is:
  \[
  \dim(\sY)=\lambda_c+K_3+\cdots+K_c.
  \]
  Our theorem thus says that the closure of $\sH$ in
  $\Hilb_{p(t)}(\p^n)$ is an irreducible variety of dimension
  $\lambda_c+K_3+\cdots+K_c$, and in fact an irreducible component of
  $\Hilb_{p(t)}(\p^n)$ if $\dim(X_\phi)=n+b-a-1 \ge 2$.
  The following addresses \cite[Conjecture 3.2]{kleppe:low-JPAA}.
  \begin{mcorol*} 
    Let $b \le a-1$, $d \geq 1$ be integers and set $\alpha_j =
    d, \beta_i = 0$ for all $i,j$, and assume $\dim(X_\phi)\ge 2$, i.e. 
    $n+ b-a-1 \ge 2$.
    Then the map $F$ is birational.
    In particular, $\sH$ is an irreducible, generically smooth
    variety of dimension:
    \[
    \dim(\sH) = a b {n+d
      \choose{n}} - a^2-b^2+1.
    \]
\end{mcorol*}
The result above is related to \cite{faenzi-fania:palatini,tanturri:degeneracy}, where the Hilbert
scheme of Palatini scrolls is described.
Shortly before submitting our paper we learned of a preprint by
Kleppe \cite{kleppe:mannaggia}, addressing very similar questions.

\subsection{Structure of the paper}

In Section \ref{basic} we review some basic material. 
In Section \ref{normale} we calculate the sections of the normal
sheaf of $X_\phi$, whereby giving an estimate on the dimension of $\sH$.
Section \ref{fibres} contains our main lemmas on the fibres of the map
$F$, and the proof of the main result.

\subsection{Notations and conventions}
Let $\kk$ be an algebraically closed field, $n\ge 2$ be an integer and let
$V$ be an $(n+1)$-dimensional vector space over $\kk$.
We consider the projective space ${\p^n}$ of $1$-dimensional quotients
of $V$. Under this convention we have, for any $d \ge 0$, a natural identification
$\HH^0({\p^n},\OO_{{\p^n}}(d))\cong S^d V$, the $d$-th symmetric power
of $V$. 

If $Z$ is a variety equipped with a morphism $f: Z \to {\p^n}$, 
the symbol $H_Z$ will denote $f^*(H_{\p^n})$, with 
 $H_{\p^n}=c_1(\OO_{{\p^n}}(1))$.
Given a morphism $\phi$ of vector bundles on a variety $Z$, 
 $\rD_k(\phi)$ will denote the locus consisting of the points $z$
of $Z$ such that $\phi_z$ has rank at most $k$,
so $\rD_k(\phi)$ is cut locally by all $(k+1)$-minors of a matrix
defining $\phi$.

We will write $R =k[x_0,\ldots,x_n] \cong S V$ for
the polynomial ring. If $Z$ is a subvariety of $\p^n$, we denote by
$R_Z$ its coordinate ring.
We write $\tilde{A}$ for the sheafification of a module $A$ over
$R_Z$.
Given a coherent sheaf $\sE$ on $Z$ we denote by
$\HH^0_*(\sE)$ the $R_Z$-module of global sections $\bigoplus_{t \in
  \Z}\HH^0(Z,\sE(t))$. If $\phi : \sE \to \sE'$ is a morphism of
coherent sheaves over $Z$, then $\HH^0_*(\phi)$ will denote the induced
morphism of $R_Z$-modules of global sections.
We write $A^*=\Hom_{R_Z}(A,R_Z)$ and $\sE^*=\HHom_Z(\sE,\OO_Z)$.
We will denote by $\sN_{Z',Z}$ the normal sheaf of  a subvariety $Z'$ in $Z$.
The shortcut $\hh^i(X,\sE)$ will be used for $\dim_{\kk} \HH^i(X,\sE)$.
We also use $\p(\sE)$ for the projective bundle associated with
$\sE$, and $S^d \sE$ for $d$-th symmetric power of $\sE$.

\section{Basic constructions} \label{basic}

The material contained in this section is well-known,
we refer to the 
monographs \cite{northcott:finite-free, bruns-vetter:LNM,
  weyman:tract, miro-roig:determinantal}.
Let us fix some notation and keep it throughout the paper.
Given integers $b \le a-1$,  $\alpha_1 \leq \cdots \leq \alpha_a$
and $\beta_1 \leq \cdots \leq \beta_b$, we define:
\[
\sA = \bigoplus_{j=1,\ldots,a} \OO_{{\p^n}}(-\alpha_j), \qquad \sB =
\bigoplus_{i=1,\ldots,b} \OO_{{\p^n}}(-\beta_i), \qquad \phi : \sA \to \sB.
\]
We write $X_\phi$ for the first degeneracy locus of $\phi$, so
$X_\phi = \rD_{b-1}(\phi)$, with the scheme-theoretic structure given
by the $b \times b$ minors of $\phi$.
The expected codimension of $X_\phi$ in $\p^n$ is $\max(n,a-b+1)$.
If $\phi$ is general enough, the actual codimension of $X_\phi$ is at
least this number if and only if (cf. \cite{kleppe-miro-roig:dimension}):
\begin{equation}
  \label{standard}
    \mbox{$\alpha_i \ge \beta_i$, for all $i=1,\ldots,b$,} \qquad
  \mbox{and $\alpha_i > \beta_i$, for some $i=1,\ldots,b$.}
\end{equation}

\subsection{The map {\it F}}

Let $\sA$ and $\sB$ be as above.
The following lemma provides a description of the
$\G$-orbit space of $\bW$ suitable for our purpose.

\begin{lem} \label{dimY}
  There is a generically smooth irreducible variety
  $\sY$ parametrizing generic
  $\G$-orbits of $\bW$, with:
  \begin{equation}
    \label{dimension}
    \dim(\sY)= \lambda_c + K_3+\cdots+K_c.    
  \end{equation}
\end{lem}

\begin{proof}
The group $\G$ is in general non-reductive.
However, we will only be interested in some open piece of the orbit
space.
According to a result of Rosenlicht, see \cite{rosenlicht:remark}, there is a dense open subset
$\bW^\circ$ such that the quotient $\bW^\circ /\G$ is geometric.
Let us denote:
\[
\sY = \bW^\circ /\G.
\]
This is a generically smooth variety of dimension:
\[
\dim(\sY)=\dim(\bW)-\dim(\G)+\dim(\G_{\phi}),
\]
where $\G_\phi$ is the stabilizer of a general element $\phi \in \bW^\circ$.
The dimension of $\G_{\phi}$ is computed in 
\cite{kleppe-migliore-miro-roig-nagel-peterson,
  kleppe-miro-roig:dimension}, and we get \eqref{dimension}.
\end{proof}

\subsection{Cokernel of a matrix with polynomial entries}
\label{coker}

Let $\phi : \sA \to \sB$ be as above, and assume \eqref{standard}.
We  assume from now on that $X_\phi$ has codimension $c=a-b+1$, so:
\[
\dim(X_\phi)=n+b-a-1.
\]
\subsubsection{Cokernel sheaf} We define the sheaf:
\[
 \sC_\phi = \cok(\phi).
\]
Note that $\sC_\phi$ is supported on $X_\phi$. Let $i$ denote the
embedding of $X_\phi$ in ${\p^n}$. Then $\sC_\phi \cong i_*(\sL_\phi)$,
for a sheaf $\sL_\phi$ on $X_\phi$ of (generic) rank $1$.
The sheaf $\sL_\phi$ is ACM on $X_\phi$ i.e.,
$\HH^0_*(X_\phi,\sL_\phi)$ is a maximal Cohen-Macaulay module over
$R_{X_\phi}$.
In particular, $\sL_\phi$ is
reflexive, hence invertible if $X_\phi$ is integral and locally factorial.

Given the sheaf $\sL_\phi$, we will consider $c_1(\sL_\phi)$, as a divisor class in $\cl(X)$
by looking at the zero locus $D_\phi$ of $\sL_\phi(t)$
(where we choose the smallest $t \in
\Z$ such that  $\HH^0(X_\phi,\sL_\phi(t)) \neq 0$).
This locus is in fact a determinantal subvariety $X_{\phi_0}
\subset \p^n$
with $\phi_0 : \sA \to \sB_0$ and $\sB_0 = \sB/\OO_{\p^n}(-t)$.
Note also that the class of $D_\phi$ determines $\sL_\phi$. Indeed we
have an exact sequence:
\[
0 \to \sL_\phi^*(-t) \to \OO_{X_\phi} \to \OO_{D_\phi}
\to 0,
\]
and this gives back $\sL_\phi$ since $\sL_\phi ^{**} \cong \sL_\phi$.

\subsubsection{Cokernel module} In terms of $R$-modules, we define:
\[
A = \bigoplus_{j=1,\ldots,a} R(-\alpha_j), \qquad B = \bigoplus_{i=1,\ldots,b} R(-\beta_i).
\]
So $\phi$ gives a morphism $\HH^0_*(\phi) : A \to B$, whose
sheafification is $\phi$.
We define:
\[C_\phi = \cok(\HH^0_*(\phi)).\]
This is a graded $R$-module. Considered as $R_X$-module, $C_\phi$ is a maximal Cohen-Macaulay
module, and we have $C_\phi \cong \HH^0_*(\sC_\phi)$ and 
$\tilde{C}_\phi \cong \sC_\phi$, see for instance \cite{miro-roig:determinantal}.

\subsection{Determinantal subvarieties as complete intersections}

For the constructions in this subsection we refer for instance to \cite{ein:determinantal,weyman:tract}.
Consider the vector bundle $\sB = \bigoplus_{i=1,\ldots,b}
\OO_{{\p^n}}(-\beta_i)$ over the projective space ${\p^n}$ and define
the projective bundle:
\[
\sP = \p(\sB) \xr{\pi} {\p^n}.
\]
We have the relatively ample line bundle $\OO_\sB(1)$ and we let $
P = c_1(\OO_\sB(1)).$
Set $\sT_\sB$ for the relative tangent bundle.
Each $\phi \in \bW$ gives a section $s_\phi \in
\HH^0(\sP,\pi^*(\sA)^*(P))$ in view of:
\[
\bW = \Hom_{{\p^n}}(\sA,\sB) \cong \Hom_{\sP}(\pi^*(\sA),\OO_{\sP}(P)) \cong
\bigoplus_{i,j}
  S^{\alpha_j - \beta_i} V. 
\]
Denoting by $\V(s_\phi)$ the vanishing locus of $s_\phi$, we define
the complete intersection subvariety:
\[
Y_\phi = \V(s_\phi) \subset \sP.
\]

To see the relation between $X_\phi$ and $Y_\phi$, we consider the scheme
$\p(\sL_\phi)$, and we let $q$ be the natural map $\p(\sL_\phi)\to {\p^n}$.
Since $i_*(\sL_\phi)$ is a quotient of $\sB$, we have a natural closed embedding
$p : \p(\sL_\phi) \mono \sP.$
Set $P_{Y_\phi}=P|_{Y_\phi}$.

The following lemma is certainly well-known, although we haven't been
able to find the precise statement in the literature.
However, we provide a proof for the reader's convenience.

\begin{lem} \label{proj}
  Choose any $\phi \in \bW$ with $X_\phi \neq \emptyset$. Then
  $\p(\sL_\phi) \cong Y_\phi$ as schemes, and we have:
  \begin{equation}
    \label{qp}
    \sL_\phi \cong q_*(\OO_{Y_{\phi}}(P_{Y_\phi})), \qquad q(Y_\phi)=X_\phi.
  \end{equation}
  If the further degeneracy locus $\rD_{b-2}(\phi)$ is empty, then $q:\p(\sL_\phi) \to {\p^n}$
  is an isomorphism onto $X_\phi$, so $Y_\phi \cong X_\phi$.
\end{lem}

\begin{proof}
  The scheme $\p(\sL_\phi)$ consists of the pairs $([\xi],[\gamma])$
  where $\xi : V \to \kk$ represents a $1$-dimensional quotient of $V$
  and the proportionality class of $\gamma$ lies in
  $\p(\sL_{\phi,\xi})$. Since $i_*(\sL_\phi)$ 
  is defined as $\cok(\phi)$,
  we have that $\gamma$ is a quotient of $\sL_{\phi,\xi}$ fitting into:
  \[
  \sA_{\xi} \xr{\phi_{\xi}} \sB_{\xi} \to \sL_{\phi,\xi}.
  \]
  Lifting $\gamma$ to a map $\sB_{\xi} \to \kk$ (still denoted by $\gamma$), we get
  that $\gamma$ is defined on $\sL_{\phi,\xi}$ if and only if $\gamma
  \circ \phi_{\xi} = 0$.
  Clearly, we have:
  \[
  \gamma \circ \phi_{\xi} = 0 \Leftrightarrow \gamma(\phi_{\xi}(e)) = 0, \forall
  e \in \sA_{\xi}.
  \]
  Summing up we have:
  \begin{equation}
    \label{PEphi}
    \p(\sL_\phi) = \{([\xi],[\gamma]) \, | \, \gamma(\phi_{\xi}(e)) = 0, \forall
    e \in \sA_{\xi} \}.
  \end{equation}
  
  On the other hand, $Y_\phi$ consists of pairs $([\xi],[\gamma])$
  such that $\gamma$ is a quotient of $\sB_\xi$ and
  $s_\phi$ vanishes at $([\xi],[\gamma])$.
  By definition of $s_\phi$, its evaluation
  $s_{\phi,([\xi],[\gamma])}$ at a pair $([\xi,\gamma])$ is given as the composition:
  \[
  \sA_\xi \xr{\phi_\xi} \sB_\xi \xr{\gamma} \OO_\xi(P) \cong \kk.
  \]
  Therefore, we have:
  \[
  Y_\phi = \{([\xi],[\gamma]) \, | \, s_{\phi,([\xi],[\gamma])}(e)=\gamma(\phi_{\xi}(e)) = 0, \forall
    e \in \sA_{\xi} \}.
  \]
  This agrees with \eqref{PEphi}, so our first statement is proved.

  To check \eqref{qp} we note that, since $p$ is  induced by the
  projection $\sB \to i_*(\sL_\phi)$, there is an isomorphism:
  \[
  p^*(\OO_{\sP}(P)) \cong \OO_{\sL_\phi}(1).
  \]
  Clearly we have:
  \[
  q_*(\OO_{\sL_\phi}(1)) \cong \sL_\phi.
  \]
  This proves \eqref{qp}.

  To check the last statement, first note that since $\sL_\phi$ is supported on $X_\phi$ the map $q$ takes value in $X_\phi$.
  When $\rD_{b-2}(\phi) = \emptyset$, the sheaf $\sL_\phi$ is locally
  free of rank one on $X_\phi$ by \cite{pragacz:enumerative}.
  So $\p(\sL_\phi) \cong X_\phi$.
\end{proof}

\begin{eg}
  Let $\alpha_j=1$ for all $j$ and $\beta_i=0$ for all $i$.
  In this case $\phi$ corresponds to the choice of a $3$-tensor in
  $\phi \in \kk^a \ts \kk^b \ts \kk^{n+1}$.
  The variety $X=X_\phi$ is cut in $\p^n$ by the $b \times b$ minors of a matrix
  $M_\phi$ of linear forms, of size $a \times b$, while $Y=Y_\phi$ is
  a linear section of codimension $a$ in the Segre product
  $\p^{b-1}\times \p^n$.
  The map $\pi$ is the projection $\p^{b-1}\times \p^n$ onto the second factor $\p^n$,
  and its restriction to $Y$ is $q$.

  The variety $Y$ is also obtained as $\p(\sM_\phi)$, where
  $\sM_\phi=s_*(q^*(\OO_{\p^n}(1)))$, and we denote by $s$ the
  projection $Y \to \p^{b-1}$.
  The sheaf $\sM_\phi$ is presented by a matrix $N_\phi$, obtained by
  exchanging the roles of $\kk^b$ and $\kk^{n+1}$ in the expression of
  $\phi \in \kk^a \ts \kk^b \ts \kk^{n+1}$, so $N_\phi$ reads:
  \[
  \OO_{\p^{b-1}}(-1)^a \xr{N_\phi} \OO_{\p^{b-1}}^{n+1} \to \sM_\phi \to 0.
  \]

  For instance, if $a=4$, $b=3$, $n=6$, then, for general $\phi$, $X$ is a $4$-fold in
  $\p^6$ with $10$ singular points, defined by the order-$2$ minors of $M=M_\phi$.
  On the other hand, $Y$ is a smooth $4$-codimensional linear section
  of $\p^2 \times \p^6$, which is a $\p^2$-bundle over $\p^2$,
  obtained by projectivizing $\sM=\sM_\phi$ which is a rank-$3$ stable vector
  bundle on $\p^2$ with $c_1=4$.
  As such, $\sM$ splits over a general line $\ell$ of $\p^2$ as
  $\OO_\ell(1)^2 \oplus \OO_\ell(2)$. 
  This bundle is a rank-$3$ logarithmic bundle associated with $10$
  lines  $(\ell_1,\ldots,\ell_{10})$ of $\p^2$, 
  in the sense of \cite{valles:schwarzenberger-math-z}.
  The lines $\ell_i$  are jumping lines of $\sM$, in the sense that
  $\sM|_{\ell_i} \cong \OO_{\ell_i} \oplus \OO_{\ell_i}(2)^2$.
  Corresponding to the trivial summand of $\sM|_{\ell_i}$, there is a
  section the ruled $3$-fold $S_i=\p(\sM|_{\ell_i})$ which is contracted to a singular
  point of $X$ in $\p^6$.
  The sheaf $\sL_\phi$ is locally free of rank $1$ away from the $10$
  singular points of $X$ and has rank $2$ over these points.
  The corresponding divisor $D_i=q(S_i)$ are Weil divisor, which fail
  to be Cartier along the $10$ singular points.
\end{eg}

\section{Normal sheaf of a determinantal subvariety} \label{normale}

In order to study the tangent space at $[X_\phi]$ of the Hilbert
scheme $\sH$, we are now going to compute the normal sheaf of $X_\phi$, together with its
space of global sections, in the range
\eqref{numerical}, where $X_\phi = \rD_{b-1}(\phi)$ is given by 
a general morphism $\phi : \sA \to \sB$.

\begin{prop} \label{sezioni-normale}
  Let $X=X_\phi$ and let $\sN = \sN_{X,{\p^n}}$ be the normal sheaf of
  $X$ in ${\p^n}$. Assume \eqref{numerical}, $\dim(X) \ge 2$ and $c=\codim(X,\p^n) \ge 2$.
  Then we have:
  \[
  \hh^0(X,\sN) \le \lambda_c+K_3+\cdots+K_c.
  \]
\end{prop}

It will turn out that the above inequality is an equality (cf.
proof of the main result, Section \ref{fibres}).
The proof of the above proposition will follow from the next lemmas
and Proposition \ref{normalsheaf}.

\begin{lem} \label{sing}
  In the range \eqref{numerical}, for $\phi$ general in $\bW$, we
  have:
  \[\codim(\sing(X_\phi),X_\phi)\ge 3.\]
\end{lem}

\begin{proof}
  This follows from \cite[Remark 2.7]{kleppe-miro-roig:dimension}, or
  equivalently from 
  \cite{chang:bertini}. Indeed, we 
  construct Chang's filtration of the bundles 
  $\sB \supset \sB_1 \supset \cdots \supset \sB_l$ and $\sA\supset
  \sA_1 \supset \cdots \supset \sA_l$, as follows.
  We choose
  $\sA_1 = \oplus_{j=1}^{r_1} \OO_{\p^n}(-\alpha_j)$ with
  $r_1 = \max \{j | \alpha_{a+1-j} \ge \beta_b\}$ and 
  $\sB_1 = \oplus_{i=1}^{s_1} \OO_{\p^n}(-\beta_i)$ with $s_1 =
  \max\{i|\beta_{b-i+1} > \alpha_{a-r_1}\}$.
  Iterating this procedure we get the desired filtration, and, since
  we are assuming $a - b +1 \ge 2$, we 
  obtain by the main theorem of \cite{chang:bertini} that 
  $\codim(X_\phi)=a-b+1$ and 
  $\codim(\sing(X_\phi),X_\phi)\ge 3$.
\end{proof}

\begin{lem}
  Set $\sC = \sC_\phi$. Then, assuming \eqref{standard}, we have:
  \[
  \hh^0({\p^n},\EExt^1_{{\p^n}}(\sC,\sC)) \leq \lambda_c+K_3+\cdots+K_c.
  \]
\end{lem}

\begin{proof}
  Keep in mind that $X=X_\phi$ is an integral ACM subvariety of
  ${\p^n}$, i.e., $R_X$ is a Cohen-Macaulay graded ring. Call $i:X \mono \p^n$ the embedding.
  Since $\dim(X)\ge 2$, the fact that $X$ is ACM implies
  $\HH^1(X,\OO_X(t))=0$ for all $t\in \Z$, and also
   $\HH^0(X,\OO_X) \cong \kk$.
  Further, recall $\sC \cong i_*(\sL)$, where $\sL=\sL_\phi$ is an ACM rank-$1$
  sheaf on $X$. Hence $\HH^1(X,\sL(t))=0$ for all $t\in \Z$.
  Set $\sK = \ker(\phi)$. We have an exact sequence:
  \begin{equation}
    \label{lunga}
    0 \to \sK \to \sA \to \sB \to \sC \to 0.
  \end{equation}

By \cite[Lemma 3.2]{kleppe-miro-roig:dimension}, we have:
\[
\HHom_{{\p^n}}(\sC,\sC) \cong i_*(\OO_{X}).
\]
Therefore, applying $\HHom_{{\p^n}}(-,\sC)$ to \eqref{lunga}, we
obtain the long exact sequence:
\[  0 \to i_*(\OO_{X}) \to \sB^* \ts \sC\xr{\psi} \sA^* \ts
  \sC \to \sF \to 0,\]
where the sheaf $\sF$, defined by the sequence above, is
supported on $X$. We have:
\[
\EExt^1_{{\p^n}}(\sC,\sC)  \subset \sF.
\]
  Then we have:
  \[
  \hh^0({\p^n},\EExt^1_{{\p^n}}(\sC,\sC)) \leq \hh^0(\p^n,\sF),
  \]
  and we  want to show:
  \[
  \hh^0(\p^n,\sF) = \lambda_c+K_3+\cdots+K_c.
  \]
  Let us assume for the moment that the following claim holds:
  \begin{claim} \label{annullo}
    Whenever $\dim(X) \ge 2$, we have $\HH^1(X,\im(\psi))=0$.
  \end{claim}
  Set $f(t)=\hh^0(X,\sL(t)) = \hh^0({\p^n},\sC(t))$.
  Assuming the above claim, we can write:
  \begin{equation}
    \label{base-conto}
    \hh^0(\p^n,\sF)=\sum_{j=1,\ldots,a} f(\alpha_j)-\sum_{i=1,\ldots,b} f(\beta_i)+1.    
  \end{equation}

The Buchsbaum-Rim complex associated with $\HH^0_*(\phi)$ gives the graded free
resolution of $C_\phi$:
\begin{align}
  \label{BR}
  0 \to &\wedge^{a} A \otimes S^{c-2} B^{*}(\beta) \to \cdots
  \to \wedge^{b+s+1} A \otimes S^{s} B^{*}(\beta) \to \cdots  \\
  \nonumber \cdots & \to  \wedge^{b+1} A(\beta) \to A \to B \to C_\phi \to 0,
\end{align}
where we set $\beta = \sum_{i=1,\ldots,b} \beta_i$. Put $\sE_s = \wedge^{b+s+1} \sA\otimes S^{s}(\sB^{*}) (\beta)$.
Sheafifying \eqref{BR} and computing global sections we get:
  \begin{align} \label{ft}
     f(t) = & \sum_{i=1,\ldots,b}{{n-\beta_{i}+t}\choose
      n}-\sum_{r=1,\ldots,a}{{n-\alpha_r+t} \choose n}+ \\
    \nonumber + & \sum_{s=0}^{c-2}(-1)^s \hh^0(\p^n,\sE_s(t)).
  \end{align}
  Further, it is easy to compute $\hh^0(\p^n,\sE_s(t))$ as:
  \begin{equation} \label{Et}
     \sum_{\substack{1\le i_1<\cdots <i_{a-b-s-1}\le a\\
        1\le j_1\le\ldots \le j_s\le
        b}}{n-\ell+\alpha_{i_1}+\cdots+\alpha_{i_{a-b-s-1}}+\beta_{j_1}+\cdots
      +\beta_{j_s}+t \choose n}.
  \end{equation}
  Note that, in view of \eqref{standard}, the upper term appearing in the
  binomial coefficient above is strictly bounded above by $-\alpha_{b+1}+t+n$, hence 
  the binomial coefficient vanishes for $t=\beta_i$, and $t=\alpha_{i-1}$ with $i\le {b+1}$.
  Then, combining \eqref{Et} and \eqref{ft}, we get an expression for
  $f(t)$, hence for $\hh^0(\p^n,\sF)$ in view of
  \eqref{base-conto}.
  Recalling the definition of $\lambda_c$ from
  \eqref{lambdac} and of the
  $K_i$'s from \eqref{Ki},  one now easily gets the desired expression
  for $\hh^0(\p^n,\sF)$.
\end{proof}

\begin{proof}[Proof of Claim \ref{annullo}]
  If $\dim(X)\ge 3$ the vanishing is clear since $\HH^1(X,\sL(t))=0$
  for all $t\in \Z$ and $\HH^2(X,\OO_X)=0$. So we only have to prove
  the vanishing for $\dim(X) = n+b-a-1 = 2$. 
  Set $Y=Y_\phi$ and recall that $X=q(Y)$.
  We get the sequence defining
  $\im(\psi)$ applying $q_*$ to:
  \begin{equation}
    \label{reltan}
  0 \to \OO_Y \to \pi^* \sB^* \ts \OO_Y(P_Y) \to (\sT_\sB)|_{Y} \to 0,    
  \end{equation}
  where $\sT_\sB$ is the relative tangent bundle of $\pi$.
  So we want:
  \[    \HH^1(Y,(\sT_\sB)|_{Y})=0.   \]

  To obtain this vanishing, we look at the Koszul complex of $s_\phi$
  and we tensor it by $\sT_\sB$.
  The $k$-th term of this complex is $\pi^*(\wedge^k \sA)\ts \sT_\sB(-kP) $.
  Then it suffices to show:
  \begin{equation}
    \label{annullo3}
    \HH^k(\sP,\sT_\sB((1-k)P + tH_\sP)) = 0, \,\, \mbox{for
      $k=1,\ldots,a+1$ and $\forall t \in \Z$.}
  \end{equation}
  
  The sequence \eqref{reltan} is the restriction of the relative Euler  sequence   twisted by $\OO_\sP(-P)$:
\begin{equation}
  \label{tangente-relativo}
  0 \to  \OO_\sP(-P)  \to \pi^*(\sB^*) \to \sT_\sB(-P) \to 0.
\end{equation}

Recall that we set $\beta = \sum_{i=1,\ldots,b} \beta_i$. We will use the natural isomorphisms (see for instance \cite[Exercise 8.4, pg 253]{hartshorne:ag}):
\begin{equation}
  \label{insomma}
    \RR^j \pi_*(\OO_\sP(1-\ell) P) \cong \left\{
    \begin{array}{ll}
        S^{1-\ell} \sB, &   \mbox{for $j=0$ and $\ell \le 1$,}\\ 
        (S^{\ell-1-b}\sB)^{*}(\beta), & \mbox{for $j=b-1$, $\ell \ge b+1$,}\\
        0 & \mbox{for $j \ne 0,b-1$}. 
    \end{array}
\right.
\end{equation}

Applying direct image functors $\RR^j \pi_*$ to
\eqref{tangente-relativo} we easily obtain the vanishing of $\RR^j
\pi_*(\sT_\sB((1-\ell) P))$ for $1\le j \le b-3$ (and obviously for $j
\ge b$).

Now we need to compute this sheaf for $j=b-1$ and $j=b-2$. 
For $j=b-2$, we show that it vanishes except for $\ell=b+1$, and for
$j=b-1$ we show that it is a direct sum of line bundles and non-zero
only for $\ell \ge b-3$. In fact $\RR^{j}
\pi_*(\sT_\sB((1-\ell) P))$ are the kernel and cokernel, respectively
for $j=b-2$ and $j=b-1$, of the induced morphism:
\begin{equation}
  \label{higherdirectImage}
 \RR^{b-1} \pi_{*}(\OO_\sP((1-\ell)P) \to \sB^* \otimes \RR^{b-1} \pi_{*}(\OO_\sP((2-\ell)P)).
\end{equation}

We deduce  $\RR^{b-2} \pi_{*}(\sT_\sB((1-\ell)P)=0$ for
$\ell\le b$. For higher $\ell$, using \eqref{insomma}, the morphism \eqref{higherdirectImage} takes
the form:
\begin{equation}
  \label{higherdirectImage'}
  (S^{\ell-1-b}\sB)^{*}(\beta) \to \sB^* \otimes (S^{\ell-2-b}\sB)^{*}(\beta).
\end{equation}
For $\ell = b+2$, this map is an isomorphism, and for all $\ell \ge
b+3$ it is injective, since it  is just the twisted dual of the obvious multiplication map:
\[
\sB \otimes S^{\ell-2-b}\sB \to S^{\ell-1-b}\sB.
\]
Therefore, from injectivity of \eqref{higherdirectImage'} we deduce
$\RR^{b-2} \pi_{*}(\sT_\sB((1-\ell)P)=0$ for $\ell\ge b+2$.
We have proved that $\RR^{b-2} \pi_{*}(\sT_\sB((1-\ell)P))=0$ except
for $\ell=b+1$. Moreover in this case we get $\RR^{b-2}
\pi_{*}(\sT_\sB(-bP)) \cong \OO_{\p^n}(\beta)$.
Also, $\RR^{b-1} \pi_*(\sT_\sB((1-\ell)P)) = 0$ for $\ell \le b+2$ while for $\ell \ge b+3$ we get that $\RR^{b-1} \pi_*(\sT_\sB((1-\ell)
P))$ is the bundle $(S^{\ell-2-b,1} \sB)^*(\beta)$ obtained by
plethysm (cf. \cite{weyman:tract}),
which is also a direct sum of line bundles.

We can now conclude by the Leray spectral sequence:
\begin{equation}
  \label{leray}
\HH^i(\p^n,\RR^j \pi_*(\sT_\sB((1-i-j) P))) \Rightarrow \HH^{i+j}(\sP,\sT_\sB((1-i-j) P)).
\end{equation}
Indeed, by the previous analysis, the only terms contributing to 
$\HH^{i+j}(\sP,\sT_\sB((1-i-j)P + tH_\sP))$ appear for $j=b-2$ or $j=b-1$.
Recall also that we only have to treat the case $a = n+b-3$ and
$a-b+1\ge 2$, i.e., $n \ge 4$.

Looking at the case $j=b-1$, we have said that $\RR^{j}
\pi_*(\sT_\sB(1-i-j) P))$ is a direct sum of line bundles for $i+j\ge
b+3$, or zero for $i+j\le b+2$. So we may assume $i \ge b+3-j=4$.
Also, $i+j \le a+1$ implies $i \le a-b+2=n-1$, so 
$\HH^i(\p^n,\RR^{j}\pi_*(\sT_\sB(1-i-j) P))=0$ in this range.
Then, \eqref{annullo3} follows from \eqref{leray}.

In case $j=b-2$, in order for $\RR^{j} \pi_*(\sT_\sB(1-i-j) P))$ to be non-zero
(and hence isomorphic to $\OO_{\p^n}(\beta)$) we
must have $i+j=b+1$, which implies $i=3$. Hence $n\ge 4$ gives $i\le
n-1$ and 
therefore $\HH^i(\p^n,\OO_{\p^n}(\beta))=0$.
So
\eqref{annullo3} holds again by \eqref{leray}.
\end{proof}

\begin{prop} \label{normalsheaf}
  Assuming \eqref{numerical}, we have:
  \[
  i_*(\sN) \cong \EExt^1_{{\p^n}}(\sC,\sC).
  \]
\end{prop}

\begin{proof}
  We recall the identification $\sC \cong i_*(\sL)$, and the natural isomorphism:
  \[
  i_*(\sN) \cong \EExt^1_{{\p^n}}(i_*(\OO_X),i_*(\OO_X)).
  \]
  We have to provide an isomorphism of the right-hand-side with
  $\EExt^1_{{\p^n}}(\sC,\sC)$.
  In order to obtain it, we consider the cohomological spectral sequence:
  \begin{equation}
    \label{spettrale}
  E_2^{p,q} = \EExt^p_{{\p^n}}(i_*(\OO_X),i_*(\EExt^q_{X}(\sL,\sL))) \Rightarrow \EExt^{p+q}_{{\p^n}}(\sC,\sC).
  \end{equation}
  Let us postpone to the end of the proof the explanation for this
  formula, and assume it for now.

  Since we have seen that $\HHom_{X}(\sL,\sL) \cong \OO_X$,
  \eqref{spettrale} this gives the required isomorphism once we prove $\EExt^1_{X}(\sL,\sL)=0$.
  In order to show this vanishing, we recall that $\sL \cong
  q_*(\OO_Y(P_Y))$, where we have set $Y=Y_\phi$.
  Projection formula provides a natural isomorphism:
  \[
  \EExt^1_{X}(\sL,\sL) \cong  \EExt^1_{Y}(q^*(\sL),\OO_Y(P_Y)).
  \]
  To show that the right-hand-side vanishes, we write the
  relative dual Euler sequence:
  \begin{equation}
    \label{tauto}
    0 \to \Omega_\phi(P_{Y_\phi}) \to q^*(\sL_\phi) \to \OO_{Y_\phi}(P_{Y_\phi}) \to 0,
  \end{equation}
  where the relative cotangent sheaf $\Omega_\phi$ is defined as the kernel
  of the canonical surjection above,
  and is supported on the locus in $Y_\phi$  blown down
  by $q$.
  We apply
  $\HHom_{Y}(-,\OO_Y(P_Y))$ to the exact sequence \eqref{tauto}.
  Clearly $\EExt^1_{Y}(\OO_Y(P_Y),\OO_Y(P_Y))=0$ for $\OO_Y(P_Y)$ is
  locally free.
  Further, $q$ is a birational surjective morphism, and the support 
  $Z_\phi$ of $\Omega_\phi$ lies over $\rD_{b-2}(\phi) \subset \sing(X_\phi)$, where
  the fibres of $q$ are generically contained in a $\p^1$. Then we have
  $\codim(Z_\phi,Y_\phi) \ge \codim(\sing(X_\phi),X_\phi)-1 \ge 2$ by
  Lemma \ref{sing}.
  Then we get $\EExt^1_{Y}(\Omega_\phi(P_{Y_\phi}),\OO_Y(P_Y))=0$, and so
  $\EExt^1_{Y}(q^*(\sL),\OO_Y(P_Y))=0$ and we are done.

  Finally, let us prove \eqref{spettrale}.
  Let $\sE$ be a coherent sheaf on $X$.
  We consider the functors $\Psi=\HHom_{\p^n}(i_*(\OO_X),i_*(-)):
  \Coh(X) \to \Coh(\p^n)$
  and $\Phi=\HHom_{X}(\sE,-) : \Coh(X) \to \Coh(X)$.
  The composition $\Psi \circ \Phi$, applied to a coherent sheaf $\sG$
  on $X$, is:
  \begin{equation}
    \label{iso}
  \HHom_{\p^n}(i_*(\OO_X),i_*(\HHom_{X}(\sE,\sG))) \cong   \HHom_{\p^n}(i_*(\sE),i_*(\sG)).    
  \end{equation}
  To see this isomorphism, we can work locally and replace the map $i$
  with the closed embedding $\Spec(A) \to \Spec(B)$
  induced by a surjective map of $\kk$-algebras $B \to A$, so that
  $\sE$, $\sG$ should be replaced with finitely generated modules $M$,
  $N$ over $A$.
  Then, to prove \eqref{iso} we have to check:
  \[
   \Hom_{B}(M,N) \cong \Hom_{B}(A,\Hom_{A}(M,N))).
  \]
 To do this, it suffices to send a $B$-morphism $u:M\to N$
 to the map $1_A \to u$, and one easily sees that this gives the
 desired isomorphism of $B$-modules.

  In this setting, Grothendieck's spectral sequence associated with the
  composition of the two left-exact functors $\Psi$ and $\Phi$,
  applied after replacing $\sE=\sG=\sL$ and recalling that $\sC = i_*(\sL)$, gives
  the required formula \eqref{spettrale} in view of \eqref{iso}.
\end{proof}

The previous lemmas, together with Proposition \ref{normalsheaf},
suffice to prove Proposition \eqref{sezioni-normale}. 

\begin{eg}
  Going back to our example of a matrix of size $3 \times 4$ over $\p^6$, we see
  that $\Omega_\phi(P_{Y_\phi})$ is 
  the direct sum of the $\OO_{\ell_i}(-1)$ for $i=1,\ldots,10$, and as
  such is supported in codimension $3$ in $Y$.
\end{eg}

\section{Fibres of the map {\it F}} \label{fibres}

We will prove here our main result.
Again we assume \eqref{numerical} and we let
 $\phi$ be a morphism $\phi:\sA \to \sB$ such that $X_\phi$ has
 codimension $c=a-b+1$ hence $\dim(X_\phi)=n-c=n-a+b-1$.

\subsection{Transitiveness on the fibres for fixed cokernel sheaves}

The following lemma shows that, once we fix the isomorphism class of
the cokernel sheaf $\sC_\phi$, the group $\mathbf{G}$ operates transitively
on the fibres of $F$.

\begin{lem} \label{transitive}
Assume $\dim(X_\phi) \ge 1$ and let $\phi'$ be a morphism $\sA \to \sB$ such that the sheaves
$\sC_\phi$ and $\sC_{\phi'}$ are isomorphic.
Then there are $g \in \Aut_{\p^n}(\sA)$ and $h \in \Aut_{\p^n}(\sB)$, such that
$h \circ \phi = \phi' \circ g$.
\end{lem}

\begin{proof}
  We have $\sC_\phi \cong \sC_{\phi'}$ and $C_{\phi'} \cong
  C_\phi$.
  The Buchsbaum-Rim complexes associated with $\HH^0_*(\phi)$ and $\HH^0_*(\phi')$ give
  graded free resolutions of $C_{\phi'} \cong C_\phi$.
  If the resolution associated with $\phi$ is minimal, then so is the
  one associated with $\phi'$, since these resolutions share the same
  Betti numbers. By the uniqueness of the minimal graded free
  resolution (see e.g. \cite{eisenbud:commutative-algebra}), we
  get the desired maps $g \in \Aut_{\p^n}(\sA) \cong \Aut_R(A)$, $h \in
  \Aut_{\p^n}(\sB)\cong \Aut_R(B)$ with $h \circ \phi = \phi' \circ g$.

  If the resolution associated with $\phi$ is not minimal, then we have
  $A \cong A_1 \oplus A_2$ and $B \cong B_1 \oplus B_2$, hence a block
  decomposition of $\phi$:
  \[
   \phi = \phi_1 \oplus \phi_2 = 
   \begin{pmatrix}
     \phi_1 & 0 \\
     0 & \phi_2
   \end{pmatrix},
  \]
  where $\phi_1 : A_1 \to B_1$ is minimal and $\phi_2 : A_2 \to B_2$
  is an isomorphism.
  Again by the uniqueness of the minimal graded free
  resolution, we get a decomposition
  $\phi' = \phi'_1 \oplus \phi'_2$ and isomorphisms 
  $g_1 \in \Aut_R(A_1)$, $h_1 \in \Aut_R(B_1)$
  with $h_1 \circ \phi_1 = \phi_1' \circ g_1$.
  Then we can take $g_2 = \id_{A_2}$, and $h_2 = \phi_2' \circ
  \phi_2^{-1}$, so setting $g = g_1 \oplus g_2$ and $h = h_1 \oplus
  h_2$ gives $h \circ \phi = \phi' \circ g$.
\end{proof}

\subsection{Finiteness and uniqueness of determinantal representations}

We start with two lemmas that account for the finiteness and the
uniqueness of the fibre of the map $F$, i.e. of determinantal
representations of a given subvariety $X=X_\phi$ of $\p^n$.
This will lead to the proof of our main result.

\begin{lem} \label{finito}
  If $\dim(X_\phi)\ge 1$,
  then, up to $\G$-action, there
  are finitely many $\phi' \in \bW$ such that:
  \[
  \sC_\phi \not \cong \sC_{\phi'}, \qquad X_\phi = X_{\phi'}.
  \]
\end{lem}

\begin{proof}
  Let $\phi'$ be a morphism $\sA \to \sB$ such that $X_\phi =
  X_{\phi'}$ and set $X=X_\phi$.
  Then $\phi'$ defines a cokernel sheaf $\sC_{\phi'} \cong
  i_*(\sL_{\phi'})$.
  Also, $X$ is the image via a map $q'$ of $Y_{\phi'}=\V(s_{\phi'})$,
  according to Lemma \ref{proj}, and we have $\sL_{\phi'} \cong
  q_*(\OO_{Y_{\phi'}}(P'_{Y_{\phi'}}))$.

  Now comes an important point to obtain our result, namely the
  canonical class of $X_\phi$.
  We set $c_1(\sL_{\phi})=P_{X}$, $c_1(\sL_{\phi'})=P'_{X}$, as an element of $\cl(X)$,
  and we define the divisor class $H_\sP = \pi^*(H_{{\p^n}})$ and its restriction $H_{Y_\phi}$ to $Y_\phi$.
  We have (see e.g. \cite{bruns-vetter:LNM}):
  \[
    K_{Y_\phi} \cong (\ell-n-1)H_{Y_\phi}+(a-b)P_{Y_\phi},\quad
     c_1(\omega_{X_\phi}) \cong (\ell-n-1)H_{X_\phi}+(a-b)P_{X_\phi}.
  \]
  Therefore:
  \[
  (\ell-n-1)H_{X}+(a-b)P_{X} \equiv (\ell-n-1)H_{X}+(a-b)P'_{X}, \qquad  \mbox{in $\cl(X)$}.
  \]
  Pulling back to $Y_\phi$ we get the equality:
  \[
  (a-b)(P_{Y_\phi}-q^*(P'_{X})) \equiv 0, \qquad  \mbox{in $\cl(Y_\phi)$}.
  \]

  We observe that $q^*$ gives an isomorphism between $\cl(Y_\phi)$ and
  $\cl(X)$, for $q$ is biregular outside a closed subset of codimension
  at least $2$. Moreover, $P_{Y_\phi}$ is a Cartier divisor of $Y_\phi$ so
  the above equality takes place in $\Pic(Y_\phi)$.
  Note that $a-b \neq 0$ by hypothesis, and that $\Pic(Y_\phi)$ has only finitely
  many points of order $a-b$.
  Indeed, this is clear if $\dim(Y_\phi)=1$ for $\Pic(Y_\phi)$ is then
  smooth, and also if $\dim(Y_\phi) \ge 2$ since in this case
  $\HH^1(Y_\phi,\OO_{Y_\phi}) = \HH^1(X,\OO_X) = 0$, so
  that $\Pic(Y_\phi)$ is reduced to the Néron-Severi group, which is
  finitely generated.

  Then, recalling that $c_1(\sL_{\phi'})$
  determines $\sL_{\phi'}$ (see Section \ref{coker}),
  we get that there are only finitely many 
  ways to choose the isomorphism class of $\sL_{\phi'}$ in such a way that the above equation is satisfied.
  In other words, there are $\phi_1,\ldots,\phi_r$
  such that $\sC_{\phi_i} \not\cong \sC_{\phi_j}$ if $i\neq j$,
  and such that for any other $\phi_0 \in \bW$ we have $\sC_{\phi_0} \cong \sC_{\phi_i}$
  for one $i=1,\ldots,r$.
  By Lemma \eqref{transitive}, this $\phi_0$ is taken to $\phi_i$ by
  the action of $\G$, which proves our claim.
\end{proof}

\begin{lem} \label{iniett}
  Let $\phi$ be a general element of $\Hom_{\p^n}(\sA,\sB)$, and assume:
  \[
   \alpha_1 > \beta_b, \qquad  \dim(X_\phi) \geq 2, \qquad c = a-b+1 \geq 2.
  \]
  Then, given another morphism $\phi'$ with $X_\phi = X_{\phi'}$, we
  have $\sC_\phi \cong \sC_{\phi'}$.
\end{lem}

\begin{proof}
  Let $X=X_\phi$, and consider the divisor $P_X$ associated with
  $\sC_\phi$.
  In view of the considerations of the previous lemma,
  any $\phi'$ such that $X_{\phi'}=X$ gives a divisor 
  class $P'_X$ such that:
  \[
  (a-b)(P_{Y_\phi}-q^*(P'_{X})) \equiv 0, \qquad  \mbox{in $\Pic(Y_\phi)$}.
  \]
  
  Again $a-b \neq 0$, so that $P_{Y_\phi}-q^*(P'_{X})$ is of order $a-b$ in
  $\Pic(Y_\phi)$. Now $Y_\phi$ is the complete intersection in
  $\sP$, of $a$ (Cartier) divisors, whose classes are:
  \[
  (\alpha_1 H_\sP + P_\sP,\ldots,\alpha_a H_\sP + P_\sP).
  \]
  Under the hypothesis $\alpha_1>\beta_b$, all the divisors
  $\alpha_j H_\sP + P_\sP$ are very ample on $\sP$, indeed the direct
  image $\pi_*(\OO_\sP(\alpha_j H_\sP + P_\sP))$ decomposes as a
  direct sum of positive line bundles.
  Therefore we can argue, by Grothendieck-Lefschetz theorem, that
  $Y_\phi$ is smooth and $\Pic(Y_\phi)$ is
  torsion-free provided that $\dim(Y_\phi)=\dim(X) \ge 2$ (see for instance
  \cite{badescu:picard}).
  Then $q^*(P'_X)=P_{Y_\phi}$ hence $\sC_\phi \cong \sC_{\phi'}$ as in
  the previous proof.
\end{proof}

\subsection{Proof of the main result} We are now in position to prove
our main theorem.

  The map $F$ is defined on a dense open subset of $\sY$ as soon as
  there is $\phi \in \bW$ such that $X_\phi$ has codimension
  $a-b+1$ unless $X_\phi$  is empty, and this is ensured by \eqref{standard}, which is clearly an open condition.
  By hypothesis we have $\dim(X_\phi) =n+b-a-1 \ge 1$, so we can apply Lemma \ref{finito}.
  We get that there are only finitely many $\G$-orbits in the
  inverse image of a given point $[X_\phi] \in \sH$, so that $F$ is
  generically finite, which proves \eqref{main-i}.

  Next, assuming $\dim(X_\phi) =n+b-a-1 \ge 2$, by Proposition
  \ref{sezioni-normale} and Lemma \ref{dimY} we have $\dim
  \sT_{[X_\phi],\sH} \le \dim(\sY)$, 
  and $\sY$ is irreducible and generically smooth.
  Since $F$ is generically finite, the image of $F$ also has
  dimension $\dim(\sY)$, so $\dim \sT_{[X_\phi],\sH} = \dim(\sH)$ (i.e.
  the inequality in Proposition \eqref{sezioni-normale} is an equality).
  Hence $\sH$ is generically smooth and $F$ is dominant on $\sH$, so
  $\sH$ is also irreducible.
  This proves  \eqref{main-ii}.

  Finally, assuming 
  $\alpha_1>\beta_b$ and $\dim(X_\phi) \geq 2$, $c =
  a-b+1 \geq 2$, we can apply Lemma \ref{iniett}.
  Then, by Lemma \ref{transitive} the action of $\mathbf{G}$ is
  generically transitive on the set of determinantal representations of
  $X_\phi$, so $F$ is generically injective. Still $F$ is dominant by
  part \eqref{main-ii}, hence $F$ is birational.

\bibliographystyle{amsalpha}
\bibliography{fania-faenzi.determinantal}

\vfill

\end{document}